\documentclass{amsart}
\pdfoutput=1

\newcommand{\R}{\mathbb{R}}
\newcommand{\C}{\mathbb{C}}

\newcommand{\Ccal}{\mathcal{C}}

\newcommand{\las}{\mathrm{las}}
\newcommand{\spann}{\mathrm{span}}
\newcommand{\Sym}{\mathrm{Sym}}

\newcommand{\Ort}[1]{\mathrm{O}(#1)}
\newcommand{\GL}[1]{\mathrm{GL}({#1})}
\newcommand{\U}[1]{\mathrm{U}(#1)}

\newcommand{\dU}[1]{\mathfrak{u}(#1)}

\usepackage{tikz}
\usepackage{diagbox}
\usepackage{amssymb}
\usepackage{microtype}
\usepackage{pgfplots}
\pgfplotsset{compat=1.17} 
\usetikzlibrary{decorations.pathreplacing}
\usepackage[pdftex,colorlinks,citecolor=black,linkcolor=black,urlcolor=black,bookmarks=false]{hyperref}
\def\MR#1{\href{http://www.ams.org/mathscinet-getitem?mr=#1}{MR#1}}

\usepackage{doi}
\usepackage{booktabs}
\usepackage[c2]{optidef}
\usepackage{ytableau}
\usepackage{algorithm}
\usepackage{algpseudocode}
\usepackage[tableposition=above]{caption}
\newtheorem{theorem}{Theorem}[section]
\newtheorem{proposition}[theorem]{Proposition}
\newtheorem{conjecture}[theorem]{Conjecture}

\theoremstyle{remark}

\theoremstyle{definition}

\numberwithin{equation}{section}
\numberwithin{table}{section}
\numberwithin{figure}{section}
\title[The Lasserre hierarchy for equiangular lines with a fixed angle]{The Lasserre hierarchy for equiangular\\ lines with a fixed angle}

\author{David de Laat}
\address{D.\ de Laat, Delft Institute of Applied Mathematics\\
Delft University of Technology\\ Delft, The Netherlands} \email{d.delaat@tudelft.nl}

\author{Fabr\'\i cio Caluza Machado}
\address{F.C.\ Machado, Instituto de Matem\'atica e Estat\'\i stica, Universidade de S\~ao Paulo, Rua do 
Mat\~ao 1010, 05508-090 S\~ao Paulo/SP, Brazil.}
\email{fabcm1@gmail.com}

\author{Willem de Muinck Keizer}
\address{W.H.H.\ de Muinck Keizer, Delft Institute of Applied Mathematics\\
Delft University of Technology\\ Delft, The Netherlands} \email{w.h.h.demuinckkeizer@tudelft.nl}

\date{4 September, 2023} 

\begin{document}

\begin{abstract}
We compute the second and third levels of the Lasserre hierarchy for the spherical finite distance problem. A connection is used between invariants in representations of the orthogonal group and representations of the general linear group, which allows computations in high dimensions. We give new linear bounds on the maximum number of equiangular lines in dimension $n$ with common angle $\arccos \alpha$. These are obtained through asymptotic analysis in $n$ of the semidefinite programming bound given by the second level.
\end{abstract}

\subjclass[2020]{90C22, 52C17}

\maketitle

\tableofcontents

\section{Introduction}

In discrete geometry, linear programming bounds are important for bounding the quality of geometric configurations \cite{delsarte77,MR1181534,MR1973059,bachoc09,cohn07}.
These bounds have been used to show that constructions are optimal, for example the sphere packings coming from the $\mathsf{E}_8$ and Leech lattices \cite{MR3664816,MR3664817}. Moreover, the best known asymptotic bounds for binary codes \cite{MR439403}, spherical codes and sphere packing densities \cite{KL78,MR3229046}, and certain problems in Euclidean Ramsey theory \cite{MR3341578,MR4439455} are derived from the linear programming bounds. 

However, for many instances the linear programming bounds are not sharp and research is being done into semidefinite programming bounds. Initiated by Schrijver for binary codes, and Bachoc and Vallentin for the kissing number problem, three-point semidefinite programming bounds have been developed which take into account interactions between triples of points, as opposed to pairs of points for the linear programming bounds \cite{MR2236252,bachoc08,MR2947943,cohnlaatsalmon}. For the equiangular lines problem, this has been generalized to a hierarchy of $k$-point bounds \cite{deLaat2021}. For many problems, these higher-order bounds lead to significant improvements, including new optimality proofs via sharp bounds; see, e.g., \cite{MR2469257,MR2947943,MR4263438}. However, until now, no new asymptotic results in the dimension have been derived from these  bounds.

The Lasserre hierarchy \cite{lasserre01}  and the dual sum-of-squares approach by Parrilo \cite{Par00} are important for obtaining bounds for hard problems in combinatorial optimization; see, e.g., \cite{MR1997246,MR2952510,MR3654194}. De Laat and Vallentin have generalized the Lasserre hierarchy for the independent set problem to a continuous setting so that it can be applied to problems in discrete geometry \cite{laat15}. The first level of the hierarchy reduces to the linear programming bound. The second level, however, is a $4$-point bound from which the $k$-point bound with $k=4$ mentioned above can be derived by removing many of the constraints. Hence, the second level of the Lasserre hierarchy, although more difficult to compute, is likely to improve bounds. The second level of this hierarchy has only been computed for an energy minimization problem in dimension $3$ \cite{laat16}. We would also like to mention that an adaptation of this hierarchy for the sphere packing problem was given recently \cite{cohnsalmon}. In summary, computing the higher levels of this hierarchy is a promising approach for various problems in discrete geometry.

In this paper, we compute the second and third level of the Lasserre hierarchy for the equiangular lines problem with a fixed angle $\arccos \alpha$. This problem asks to determine the maximum number $N_\alpha(n)$ of lines through the origin in $\R^n$ such that the angle between any pair of lines is $\arccos \alpha$. Recently, there have been several breakthroughs for this problem, starting with the result by Bukh in 2016 that $N_\alpha(n)$ is at most linear in $n$ for any fixed $\alpha$ \cite{bukh16}. In 2018 Balla, Dr\"axler, Keevash, and Sudakov showed $\limsup_{n\to\infty} N_\alpha(n)/n$ is at most $1.93$ unless $\alpha = 1/3$, in which case it is $2$ \cite{balla18}. In 2021 Jiang, Tidor, Yao, Zhang, and Zhao showed
\begin{equation}\label{eq:realasymptoticslope}
N_{\alpha}(n) = \lfloor (a+1)(n-1)/(a-1) \rfloor
\end{equation}
for $\alpha = 1/a$ with $a$ an odd integer $a \ge 3$ and all sufficiently large $n$ \cite{MR4334975}.  The significance of these particular inner products $\alpha = 1/a$ lies in the fact that the equiangular lines problem without fixed angle can be solved in any dimension by computing $N_\alpha(n)$ for a given finite list of such $\alpha$; see \cite{larman77}. The linear bound by Bukh holds for all dimensions, but the slope is huge. The results from \cite{balla18} and \cite{MR4334975} hold only for $n \geq n_\alpha$ for a very large $n_\alpha$. In these results, the parameters are so large because the proofs rely on Ramsey theory. 

Asymptotically linear bounds not relying on Ramsey theory have also been given \cite{glazyrin18, balla2021}. The best current bound was given by Balla in 2021. He proved (Theorem 1 in \cite{balla2021}) for all dimensions $n$ and for $\alpha \in (0,1)$ the bound 
\begin{equation}\label{eq:ballasmalldim}
N_{\alpha}(n) \leq \frac{\sqrt{n}}{2 \alpha^3} + \frac{1+\alpha}{2 \alpha}n.
\end{equation}
Our main result is the following conjecture, which we prove for $\alpha = 1/a$ with $a = 3,5,7,9,11$.

\begin{conjecture}\label{conj:las2}
Let $\alpha \in (0,1)$. The optimal objective of the second level $\las_2(n)$ of the Lasserre hierarchy for bounding $N_\alpha(n)$ satisfies
\begin{equation}\label{eq:ourbound}
\las_2(n) \leq c_\alpha + \frac{1+\alpha}{2 \alpha}n
\end{equation}
for $n \geq n_\alpha$, where $c_\alpha$ does not depend on the dimension.
\end{conjecture}

For the values of $\alpha$ mentioned, we provide both an explicit small $c_\alpha$ and an explicit small $n_\alpha$, so we give new bounds in many dimensions not covered by (\ref{eq:realasymptoticslope}). Furthermore, because of the $\sqrt{n}$ term, we also improve on (\ref{eq:ballasmalldim}) for these $\alpha$. Our explicit linear bounds are listed in Table~\ref{table:asymptotic}. These are the first asymptotic bounds in the dimension coming from semidefinite programming. We have also computed the third level of the Lasserre hierarchy, with which we obtain new bounds for fixed dimensions. More detailed results are given in Section \ref{sec:applications}.  The third level also provides numerical evidence for the existence of linear bounds for all dimensions, of which the asymptotic slopes improve on (\ref{eq:ballasmalldim}) and our bound (\ref{eq:ourbound}). 
\begin{table}
\begin{tabular}{@{}ccc@{}} 
\toprule
$\alpha$ & $c_\alpha$ & $n_\alpha$\\
\midrule 
$1/3$ & $4$ & $13$\\
$1/5$ & $30$ & $87$\\
$1/7$ & $116$ & $261$\\
$1/9$ & $316$ & $166018$\\
$1/11$ & $699$ & $751307$\\
\bottomrule
\end{tabular}
\bigskip
\caption{We prove $N_\alpha(n) \leq c_\alpha + \frac{1+\alpha}{2\alpha} n$ for all $n \geq n_\alpha$. The constants $n_{1/9}$ and $n_{1/11}$ can also be made smaller by solving finitely many semidefinite programs.}
\label{table:asymptotic}
\end{table}

We would like to reflect briefly on how surprising Conjecture~\ref{conj:las2} is for a semidefinite programming bound. For fixed $n$, finite convergence of the Lasserre hierarchy is guaranteed, so in principle this hierarchy can be used to solve any equiangular lines problem. However, for a fixed level of the hierarchy, it could have been that the bound on $N_\alpha(n)$ becomes very bad for large $n$. Indeed, this is the behavior we see for the Delsarte and $k$-point bounds, where for any fixed $\alpha$ and $k$, the bound on $N_\alpha(n)$ grows rapidly as a function of $n$. This is a testament to the strength of the Lasserre hierarchy.

For the Lasserre hierarchy we follow \cite{laat15}. Consider the infinite graph $G$ on the vertex set $S^{n-1}$, where two distinct vertices $x$ and $y$ are adjacent  if $x \cdot y$ do not lie in some prescribed set $D$. For finite $D$ we call such $G$ a spherical finite distance graph and for the equiangular lines problem with a fixed angle we set $D = \{\pm \alpha\}$. An independent set is a subset  of the vertex set in which no two vertices are adjacent. Denote by $\mathcal{I}_t$ the  set of all independent sets of size at most $t$. This set is given a topology; see \cite{laat15}. Let $\Ccal(X)$ be the space of real-valued continuous functions on a topological space $X$, and define the operator 
$
A_t \colon \Ccal(\mathcal{I}_t \times \mathcal{I}_t)_{\mathrm{sym}} \to \Ccal(\mathcal{I}_{2t})
$
by
\[
A_tK(S) = \sum_{\substack{J,J' \in \mathcal{I}_t\\  J\cup J' = S}} K(J,J').
\]
Here $\Ccal(\mathcal{I}_t \times \mathcal{I}_t)_\mathrm{sym}$ is the space of continuous functions $K(J,J')$ that are symmetric in $J$ and $J'$. We define $\Ccal(\mathcal{I}_t \times \mathcal{I}_t)_{\succeq 0}$ to be the cone of positive kernels, which is the set of the symmetric continuous functions that satisfy
\[
\sum_{i,j=1}^N c_i c_j K(J_i, J_j) \geq 0
\]
for all $N \geq 0$, $c \in \R^N$, and $J_1,\ldots,J_N \in \mathcal{I}_t$. The Lasserre hierarchy for this problem is
\begin{mini}
{}{K(\emptyset, \emptyset)}{}{}
\label{pr:las}
\addConstraint{K \in \Ccal(\mathcal{I}_t \times \mathcal{I}_t)_{\succeq 0}}{}{}
\addConstraint{A_t K(S) \leq -1_{\mathcal{I}_{=1}}(S),\qquad}{}{S \in \mathcal{I}_{2t}\setminus\{\emptyset\},}
\end{mini}
where $1_{\mathcal{I}_{=1}}$ is the indicator function of the set $\mathcal{I}_{=1}$ of independent sets of size $1$. Let us verify that this program bounds the independence number. If $C$ is an independent set and $K$ a feasible kernel, then we have
\[
0 \leq \sum_{\substack{J,J' \in \mathcal{I}_t\\J,J' \subseteq C}} K(J,J') = \sum_{\substack{S \in \mathcal{I}_{2t}\\ S \subseteq C}} A_tK(S) \leq K(\emptyset, \emptyset) - |C|,
\]
which shows any feasible solution to \eqref{pr:las} gives an upper bound on the independence number.

To compute the second ($t=2$) and third ($t=3$) levels we make essential use of symmetry. The action of the orthogonal group $\Ort{n}$ on $S^{n-1}$ induces an action on $\mathcal{I}_t$ and $\mathcal{I}_{2t}$, and hence a linear action on $\Ccal(\mathcal{I}_t \times \mathcal{I}_t)$ and $\Ccal(\mathcal{I}_{2t})$ by $\gamma K(J,J') = K(\gamma^{-1} J, \gamma^{-1} J')$ and $\gamma f(S) = f(\gamma^{-1} S)$. By compactness of the orthogonal group, one may restrict to $\Ort{n}$-invariant kernels, which are kernels $K$ satisfying $\gamma K = K$ for all $\gamma \in \Ort{n}$. After this restriction, it is sufficient to impose the constraints $A_t K(S) \leq -1_{\mathcal{I}_{=1}}(S)$ for orbit representatives $S$ of $\mathcal{I}_{2t}\setminus\{\emptyset\}$, of which there are finitely many. Indeed, we may enumerate them by listing all matrices of size at most $2t$ that are positive semidefinite and have rank at most $n$ with ones on the diagonal and elements from $D$ elsewhere. These matrices are Gram matrices of the orbit representatives and, up to simultaneous permutations of rows and columns, are in one to one correspondence with the orbits.

The first level of the hierarchy is equal to the Lov\'asz theta prime number \cite{laat16}, and this can be reduced to the Delsarte, Goethals, Seidel linear programming bound using Schoenberg's characterization \cite{schoenberg42,bachoc09}. In this paper, we give a 
procedure for parametrizing the kernels $K \in \Ccal(\mathcal{I}_t \times \mathcal{I}_t)$ for general $t$, through which we can compute the $t$-th level of the hierarchy using semidefinite programming.

We construct positive, $\Ort{n}$-invariant kernels as follows. Under the action of $\Ort{n}$, the space $\mathcal{I}_t$ decomposes as a disjoint union of finitely many orbits $X_1,\ldots, X_N$. Fix an orbit representative $R_i$ for each orbit $X_i$, and let $H_i$ be the stabilizer subgroup of $\Ort{n}$ with respect to $R_i$. Given an irreducible, unitary representation $\pi \colon \Ort{n} \to \mathrm{U}(V)$, for each $i$ we let $e_{\pi,i,1},\ldots,e_{\pi,i,d_{\pi,i}}$ be a basis of the space of invariants
\[
V^{H_i} = \big\{ v \mid \pi(h) v = v \text{ for all } h \in H_i \big\}.
\]
For $J \in \mathcal I_t$, define $i(J)$ to be the index such that $J \in X_{i(J)}$ and let $s \colon \mathcal I_t \to O(n)$ be a function such that $s(J) R_{i(J)} = J$ for all $J \in \mathcal I_t$. For each $\pi$ let $F^\pi$ be a positive semidefinite matrix whose rows and columns are indexed by pairs $(i,j)$ with $1 \leq i \leq N$ and $1 \leq j \leq d_{\pi, i}$, where we assume only finitely many of these matrices are nonzero. The kernel $K \in \Ccal(\mathcal{I}_t \times \mathcal{I}_t)$ defined by
\begin{equation}\label{eq:kernelfourier}
K(J_1,J_2) = \sum_\pi \sum_{j_1, j_2} F_{(i(J_1),j_1), (i(J_2), j_2)}^\pi \big\langle \pi(s(J_1)) e_{\pi,i(J_1),j_1}, \pi(s(J_2)) e_{\pi,i(J_2),j_2} \big\rangle
\end{equation}
is continuous, positive, $\Ort{n}$-invariant, and does not depend on the choice of the function $s$. Moreover, any positive, $\Ort{n}$-invariant kernel can be written as a uniformly absolutely converging sequence of such kernels (in fact, as a single infinite series of the above form). A similar characterization of positive, invariant kernels goes back to Schoenberg \cite{schoenberg42} for the sphere and Bochner for general homogeneous spaces \cite{bochner41}. In \cite[Theorem 3.4.4]{laatthesis16} a variant for finitely many orbits is given.

In all applications in this paper, the vectors in $R_i$ are linearly independent. In this case, the stabilizer subgroups $H_i$ when considering the $t$-th level of the Lasserre hierarchy are isomorphic to $S(R_i) \times \Ort{n-t_i}$, where $t_i = |R_i|$ and $S(R_i)$ is a finite subgroup of $\Ort{t_i}$. To compute the inner product in \eqref{eq:kernelfourier}, we need to explicitly construct bases of $V^{H_i}$. 

Historically, a description of the invariant subspace $V^{\Ort{n-t}}$ was found when studying the Stiefel manifold $\Ort{n}/\Ort{n-t}$. In 1977, Gelbart discovered a beautiful connection between the harmonic analysis of the Stiefel manifold and the representation theory of general linear groups \cite{gelbart74}. Using weight theory, the irreducible representations of $\Ort{n}$ and of $\GL{t}$ can both be labeled by tuples $\lambda$ of integers, and he showed that the dimension of $V^{\Ort{n-t}}$ is equal to the dimension of the representation of $\GL{t}$ with the same label $\lambda$. Later Gross and Kunze \cite{grosskunze77} gave an explicit isomorphism with additional properties between $V^{\Ort{n-t}}$ and the representation of $\GL{t}$. To construct explicit bases of the invariant subspaces and to compute the inner products, we will use: bases of the representations of $\GL{t}$; the isomorphism by Gross and Kunze; and further techniques to deal with the finite groups $G$.  With this explicit description of the invariant subspaces, the kernel in \eqref{eq:kernelfourier} can be evaluated at points. This allows us to set up and compute the second and third level of the Lasserre hierarchy for spherical finite distance problems. 

In Section~\ref{sec:gltrep} and \ref{sec:Gross-Kunze} we discuss an explicit formulation of the matrix coefficients of $\GL{t}$ and the construction of Gross and Kunze. In  Section~\ref{sec:stabinv} these constructions are used to compute the invariants subspaces.  In Section~\ref{sec:sdp} an efficient semidefinite programming formulation is presented. In Section~\ref{sec:applications} we discuss computational results, where we first give improved bounds on $N_\alpha(n)$ in fixed dimensions, as shown in Figures~\ref{fig:a5} and \ref{fig:a7}, and then consider bounds for more general spherical finite distance problems. In Section~\ref{sec:asymptotics} an asymptotic analysis of the bounds is given, which leads to the proof of Conjecture~\ref{conj:las2} for the values of $\alpha$ mentioned. Finally, in Section~\ref{sec:moreasymptotics} we give a formulation for the limit  semidefinite program.

\section{Representations of the general linear group} \label{sec:gltrep}

In this section we give an explicit procedure to compute the matrix coefficients of the representations of the general linear group $\GL{t}$ over the complex numbers. We use Weyl's construction of the representations, following Chapter 15 in \cite{fulton91}. It is necessary to go through this material in some detail, since we will use the specifics of this  construction in Section \ref{sec:stabinv} and \ref{sec:moreasymptotics}.

Let $\lambda = (\lambda_1, \dots, \lambda_t)$ be a partition of $d$ with $\lambda_1 \geq \dots \geq \lambda_t \geq 0$.
The Young diagram associated with $\lambda$ consists of left-aligned rows of boxes, with $\lambda_i$ boxes in the $i$th row. For instance, the Young diagram of the partition $(4,3,1)$ is
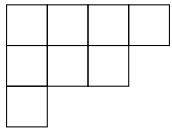
\[
\ydiagram{4,3,1}
\]
To the partition $\lambda$ we may associate two other tuples of integers: The conjugate partition $\mu$, with $\mu_j$  the length of column $j$ in the diagram of $\lambda$, and the tuple of integers $a = (a_1, \dots, a_t)$, defined by $a_i = \lambda_{i} - \lambda_{i+1}$ (setting $\lambda_{t+1} = 0$), so that $a_i$ is the number of columns with length $i$ in the Young diagram of $\lambda$. In our example, the conjugate partition is $\mu = (3,2,2,1)$ and we have $a = (1,2,1)$.

Let $\mathbb{C}^t$ be the tautological representation of $\GL{t}$, which is the representation that sends a matrix to itself, and consider the subrepresentation of $(\mathbb{C}^t)^{\otimes d}$ given by
\[
A^{a}\C^t = \Sym^{a_t}(\Lambda^t \C^t) \otimes \Sym^{a_{t-1}}(\Lambda^{t-1} \C^t) \otimes \cdots \otimes \Sym^{a_1}(\C^t).
\]
To improve readability, we will use~$\cdot$ for both the tensor and symmetric tensor products, since the length of the wedge product prevents confusion. 
To obtain an irreducible representation, we consider a quotient of $A^{a}\C^t$. Let $p$ and $q$ be the lengths of two consecutive columns of the Young diagram and let $v_1, \dots, v_p, w_1, \dots, w_q \in \C^t$. Let $r$ be an integer with $1 \leq r \leq q$. Consider the difference between an element of $A^a\C^t$ of the form
\[
x \cdot (v_1 \wedge \cdots \wedge v_p)\cdot (w_1 \wedge \cdots \wedge w_q) \cdot y
\]
and 
\begin{align*}
&x \cdot \Big(\sum (v_1 \wedge \cdots \wedge w_1 \wedge \cdots \wedge w_r \wedge \cdots \wedge v_p) \\
&\quad \quad \cdot (v_{i_1} \wedge v_{i_2} \wedge \cdots \wedge v_{i_r} \wedge w_{r+1} \wedge \cdots \wedge w_q) \Big)\cdot y,
\end{align*}
where the sum is over all $1 \leq i_1 < i_2 < \cdots < i_r \leq p$ and the elements $w_1, \dots, w_r$ are inserted at the positions $i_1, \dots, i_r$ in $v_1 \wedge \cdots \wedge v_p$. Next, consider the subspace generated by such differences. This subspace is invariant under the action of $\GL{t}$ and hence the quotient of $A^a \C^t$ by this subspace is a representation of $\GL{t}$. This representation is denoted by $W$ and the associated group homomorphism by
\begin{equation} \label{eq:rhorep}
    \rho \colon \GL{t} \to \mathrm{GL}(W).
\end{equation}
This is an irreducible polynomial representation of $\GL{t}$, and all irreducible polynomial representations of $\GL{t}$ are obtained in this way for a unique $\lambda$; see Theorem 6.3 and Proposition 15.47 in \cite{fulton91}.

We now construct a basis of $W$ using the \emph{semistandard tableaux} on the Young diagram of~$\lambda$. A tableau is obtained by placing in each box of the Young diagram an integer between $1$ and $t$. We say a tableau is semistandard if the entries in each row are nondecreasing and the entries in each column are strictly increasing. Let $T$ be a tableau and let $T(i,j)$ be the entry of $T$ in the $i$th row and the $j$th column. Define $e_T$ to be the image in $W$ of the element
\[
\prod_{j = 1}^{\lambda_1} e_{T(1,j)} \wedge e_{T(2,j)} \wedge \cdots \wedge e_{T(\mu_j,j)} \in A^{a}\C^t,
\]
where $e_1,\dots,e_t$ is the standard basis of $\C^t$. The set of all $e_T$ with $T$ a semistandard tableau on $\lambda$ is a basis of $W$ (Proposition 15.55 in \cite{fulton91}). 

We would now like to compute the matrix coefficients of the representation $\rho$ with respect to the semistandard tableaux basis. For a tableau $T$ we have
\begin{align*}
    & \rho(A) e_T = \rho(A) \prod_{j = 1}^{\lambda_1} e_{T(1,j)} \wedge e_{T(2,j)} \wedge \cdots \wedge e_{T(\mu_j,j)} \\
    & \quad = \prod_{j = 1}^{\lambda_1} A e_{T(1,j)} \wedge A e_{T(2,j)} \wedge \cdots \wedge A e_{T(\mu_j,j)} \\
    & \quad = \sum_{S} \prod_{j = 1}^{\lambda_1} A_{S(1,j),T(1,j)} e_{S(1,j)} \wedge  A_{S(2,j),T(2,j)} e_{S(2,j)} \wedge \cdots \wedge A_{S(\mu_j,j),T(\mu_j,j) } e_{S(\mu_j,j)},
    \end{align*}    
    where the sum runs over all tableaux $S$. Using multilinearity, we obtain    \begin{align*}
    \rho(A) e_T &= \sum_{S} A_{S,T} \left( \prod_{j = 1}^{\lambda_1}  e_{S(1,j)} \wedge  e_{S(2,j)} \wedge \cdots \wedge e_{S(\mu_j,j)} \right) = \sum_{S} A_{S,T} e_S,  
\end{align*}
where we define
\[
A_{S,T} = \prod_{j = 1}^{\lambda_1} \prod_{i=1}^{\mu_j} A_{S(i,j),T(i,j)}. 
\]
Let $T_1, \dots, T_m$ be the list of all semistandard tableaux on the Young diagram associated to $\lambda$. The dimension $m$ is given by (Theorem 6.3 of~\cite{fulton91})
\begin{equation}\label{eq:dimension}
m = \dim W = \prod_{1 \leq i < j \leq t} \frac{\lambda_i - \lambda_j + j - i}{j - i}.
\end{equation}
Equip  $W$ with the inner product $\langle \cdot, \cdot \rangle$ for which $e_{T_1},\ldots,e_{T_m}$ are orthonormal. For each tableau $S$ we then have
\[
e_S = \sum_{i = 1}^m \langle e_{T_i} , e_{S} \rangle e_{T_i}.
\]
The proof of Proposition 15.55 in~\cite{fulton91} may be used to give an algorithm to compute these numbers, which we describe below. The matrix coefficients of the representation $W$ are then given by
\begin{equation}\label{eq:matrixcoeffs}
\langle e_{T_i} , \rho(A) e_{T_j} \rangle =  \sum_S A_{S,T_j} \langle e_{T_i}, e_S \rangle.
\end{equation}

Let $T$ be a tableau on the Young diagram associated to $\lambda$ and $e_T$ the associated element of $W$. We encode $T$ as a list of vectors of integers, where the $i$th vector corresponds to the $i$th column of $T$. For instance the tableau 
\[
\ytableausetup{centertableaux}
\begin{ytableau}
a_1 & b_1 & c_1 \\
a_2 & b_2 & c_2 \\
a_3
\end{ytableau}
\]
will be encoded as $\{ (a_1,a_2,a_3),(b_1,b_2),(c_1,c_2) \}$. We order such tableaux using the reverse lexicographical ordering, i.e. $T' > T$ if the last entry in the last vector where $T'$ differs from $T$ is larger. 

For an integer $1 \leq s \leq \lambda_1$, let $(a_1,\dots,a_p)$ be the $s$-th column of $T$ and $(b_1,\dots,b_q)$ be the $(s+1)$-th column of $T$. For $1 \leq r \leq q$ and integers $1 \leq i_1 < \cdots < i_r \leq p$, let $\Phi(T,s,r,i_1,\ldots,i_r)$ be the tableau obtained from $T$ where we replace the first $r$ entries in column $s+1$ by $a_{i_1},\ldots,a_{i_r}$ and replace the $i_1,\ldots,i_r$ entries in column $s$ by $b_1,\ldots,b_r$. 
The algorithm below returns a vector with the coordinates of $e_T$ in the basis $e_{T_1}, \ldots, e_{T_m}$. The algorithm is recursive and terminates since at every call we  increase the position of the tableau in the reverse lexicographic order. The last step uses the quotient in the definition of $W$.

\begin{algorithm}
\caption{Algorithm to decompose a tableau into semistandard tableaux}\label{euclid}
\begin{algorithmic}
\Procedure{ssdecomp}{$T$}
\State  If there is a column in $T$ with two identical entries, return $0$.
\State Else, let $\sigma$ be the permutation of $T$ which orders each column to become strictly increasing and does not exchange elements between different columns.
\State Replace $T$ by $\sigma(T)$.
\State If $T$ is semistandard, return a vector with $\mathrm{sign}(\sigma)$ at the appropriate entry and zeros otherwise.
\State Else, find $r$ and $s$ such that $T(r,s) > T(r,s + 1)$.
\State Return $\mathrm{sign}(\sigma) \sum_{1 \leq i_1 < \dots < i_r \leq p} \text{\sc ssdecomp}(\Phi(T, s, r, i_1,\ldots,i_r))$.
\EndProcedure
\end{algorithmic}
\end{algorithm}

\section{The Gross-Kunze construction}

\label{sec:Gross-Kunze}
We see the group $\Ort{n-t}$ as a subgroup of $\Ort{n}$ by fixing the first $t$ coordinates. Given a representation $\pi \colon \Ort{n} \to \mathrm{GL}(V)$, the space of $\Ort{n-t}$-invariants is 
\[
V^{\Ort{n-t}} = \big\{v \in V \mid \pi(A) v = v \text{ for all } A \in \Ort{n-t}\big\}.
\]
In this section we give an explicit description of the $\Ort{n-t}$-invariant subspaces of the irreducible representations of $\Ort{n}$ for $2t < n$. For this we give an exposition of a construction by Gross and Kunze \cite{gross1984finite} and prove additional properties required in our setting. 

First, Gross and Kunze construct invariants in the case of the complex orthogonal group
\[
\Ort{n, \C} = \left\{ Q \in \C^{n \times n} \mid Q^{\sf T} Q = 1 \right\},
\]
whereas we are interested in the real orthogonal group $\Ort{n}$. This problem is resolved by complexifying. There is a natural inclusion $\Ort{n} \subseteq \Ort{n,\C}$ and $\Ort{n,\C}$ is the complexification of $\Ort{n}$, which means that for any smooth homomorphism $\alpha$ from $\Ort{n}$ to a complex Lie group $G$, there is a unique holomorphic homomorphism $\alpha_\C$ from $\Ort{n,\C}$ to $G$ with $\alpha(A) = \alpha_\C(A)$ for all $A \in \Ort{n}$; see, e.g., \cite[Chapter 15]{MR3025417}. 

Let $V$ be a finite dimensional, complex vector space. We consider a representation of $\Ort{n}$ on $V$ to be a smooth group homomorphism from $\Ort{n}$ to $\mathrm{GL}(V)$ and a representation of $\Ort{n,\C}$ on $V$ to be a holomorphic group homomorphism from $\Ort{n,\C}$ to $\mathrm{GL}(V)$. Holomorphic maps are in particular smooth and $\Ort{n}$ is a smooth submanifold of $\Ort{n,\C}$. Hence, the restriction of a representation of $\Ort{n,\C}$ to $\Ort{n}$ is a representation of $\Ort{n}$. By the definition of the complexification and setting $G = \mathrm{GL}(V)$, we see that any representation $\pi$ of $\Ort{n}$ is the restriction of a unique representation $\pi_\C$ of $\Ort{n,\C}$.

We now show that complexification interacts well with invariant subspaces. Since $\Ort{n-t} \subseteq \Ort{n-t, \C}$ we have
\[
V^{\Ort{n-t,\C}} \subseteq V^{\Ort{n-t}}.
\]
For the other direction, we use polar decomposition. Any matrix $A \in \Ort{n,\C}$ can be written uniquely as
\[
A = U e^{i X}
\]
with $U$ in $\Ort{n}$ and $X$ in the Lie algebra $\mathfrak{o}(n)$ of $\Ort{n}$ consisting of the skew-symmetric matrices of size $n$;  see, e.g., \cite[Proposition~15.2.1]{MR3025417}. Since $\pi_\C$ is a homomorphism we have
\[
\pi_\C(U e^{iX}) = \pi_\C(U) \pi_\C(e^{iX}) = \pi_\C(U) e^{d\pi_\C(iX)} =  \pi_\C(U) e^{id\pi_\C(X)} = \pi(U) e^{i d\pi(X)},
\]
where $d\pi_\C \colon \mathfrak o(n, \C) \to \mathfrak{gl}(n)$ and $d\pi \colon \mathfrak o(n) \to \mathfrak{gl}(n)$ are the differentials of $\pi_\C$ and~$\pi$. Here we used that $d\pi_\C$ is complex linear since $\pi_\C$ is holomorphic.
Now let $v$ be a vector in $V$ invariant under $\Ort{n-t}$. 
Then $d\pi(X)v = 0$ for any $X$ in the Lie algebra $\mathfrak o(n-t)$. By the above, for $A \in \Ort{n-t, \C}$ there are $U \in \Ort{n-t}$ and $X \in \mathfrak{o}(n-t)$ such that $A = Ue^{iX}$. Hence we have
\[
\pi_\C(A) v = \pi(U) e^{i d\pi(X)} v = \pi(U) (I + id\pi(X) + \frac{1}{2} (id\pi(X))^2 + \ldots) v = \pi(U) v = v,
\]
which shows 
\[
V^{\Ort{n-t}} \subseteq V^{\Ort{n-t, \C}}.
\]
This shows the invariant subspaces agree in the real and complex case.

The $\Ort{n-t, \C}$-invariant subspaces of the irreducible representations of $\Ort{n,\C}$ are described by Gross and Kunze in \cite{gross1984finite}. The irreducible representations are induced by representations of the general linear group. Let $(\rho,W)$ be an irreducible, polynomial representation of $\GL{t}{}$. Let $\omega$ be the complex $t \times n$ matrix
\[
\omega = \begin{pmatrix} I_t\ i I_t\ 0 \end{pmatrix},
\]
and let $\epsilon$ be the $n \times t$ matrix
\[
\epsilon = \begin{pmatrix} I_t \\ 0 \end{pmatrix}.
\]
For each $w \in W$, define a function $\phi(w) \colon \Ort{n,\C} \to W$ by
\begin{equation} \label{eq:defoforthrep1}
    \phi(w) (x) = \rho(\omega x \epsilon)w.
\end{equation}
We then define the vector space of right translates of such functions
\[
V = \spann \left\{ \pi_{\C}({\xi}) \phi(w) \mid \xi \in \Ort{n,\C},\, w \in W \right\},
\]
where $\pi_{\C}({\xi}) \phi(w) (x) = \phi(w) (x \xi)$. This space carries an action of $\Ort{n,\C}$ by right translation.       Because $h \epsilon = \epsilon$ for $h \in \Ort{n-t,\C}$, it follows that the subspace $\phi(W)$ is invariant under right translation by an element of $\Ort{n - t,\C}$. The following theorem is a special case of Theorem~7.6 in \cite{gross1984finite}. 

\begin{theorem}\label{thm:GK}
    The vector space $V$ with right translation is a finite dimensional, holomorphic representation of $\Ort{n,\C}$. Moreover, the space $V^{\Ort{n - t,\C}}$ is precisely $\phi(W)$. Every holomorphic irreducible representation of $\Ort{n,\C}$ with non-trivial $\Ort{n - t,\C}$-invariants is of this form for a unique polynomial representation $(\rho,W)$ of $\GL{t}$.
\end{theorem}
Restricting to $\Ort{n}$ gives a full description of the invariant subspaces in the real case. Henceforth, when we speak of $V$ we shall mean the representation of $\Ort{n}$ obtained by restricting the representation of $\Ort{n,\mathbb{C}}$ constructed in Theorem \ref{thm:GK}. 

We define on $V$ the inner product
\[
\langle f, g \rangle = \int_{\Ort{n}} \langle f(x), g(x) \rangle \, dx,
\]
where in the integrand any inner product on $W$ is used. By standard properties of the Haar measure, this makes the representation of $\Ort{n}$ unitary. By uniqueness of such inner products, this shows that the above integral is independent of the choice of inner product on $W$ up to multiplication by a positive real number. We now show that in the semistandard tableaux basis, the matrix coefficients of $\pi$ are real. This allows us to restrict to real matrices in the semidefinite program.

\begin{proposition} \label{prop:reality}
The numbers $\langle \phi(e_{T_i}), \pi(y) \phi(e_{T_j}) \rangle$ are real for all $y \in \Ort{n}$.
\end{proposition}

\begin{proof}
     We may define a conjugation $\overline{(\cdot)}$ on $W$ by conjugating the components of a vector in the semistandard tableaux basis. Let $\eta$ be the orthogonal matrix $ \mathrm{diag} (I_t,-I_{n - t})$. We then have $\omega \eta = ( I_t \, - i I_t \, 0 )$. Given the fact that the matrix coefficients of the representation $\rho$ in the semistandard tableaux basis are polynomials with real coefficients, we then have 
     \[
     \rho(\omega \eta x y \epsilon) \overline{v} = \overline{\rho(\omega x y \epsilon) v}
     \]
     for all $x, y \in \Ort{n}$ and all $v \in W$. Hence we have
     \begin{align*}
          \langle \phi(e_{T_i}), \pi(y) \phi(e_{T_j}) \rangle 
         &= \int_{\Ort{n}} \langle \rho(\omega  x \epsilon) e_{T_i}, \rho(\omega x y \epsilon) e_{T_j} \rangle\, dx \\
         & = \int_{\Ort{n}} \langle \rho(\omega  \eta x \epsilon) e_{T_i}, \rho(\omega \eta x y \epsilon) e_{T_j} \rangle\, dx \\
         & = \int_{\Ort{n}} \langle \overline{ \rho(\omega  x \epsilon) e_{T_i}}, \overline{\rho(\omega x y \epsilon) e_{T_j}} \rangle\, dx \\
         & = \overline{\langle \phi(e_{T_i}), \pi(y) \phi(e_{T_j})\rangle}.\qedhere
     \end{align*}
\end{proof}

\section{Invariants under the stabiliser subgroups}\label{sec:stabinv}

Recall that $\mathcal I_t$ is the set of subsets of $S^{n-1}$ of size at most $t$ with inner products in the finite subset $D$ of $[-1, 1)$. We fix a representative $R_i$ for each orbit $X_i$ of the action of $\Ort{n}$ on $\mathcal I_t$ and let $H_i$ be the stabilizer subgroup of $\Ort{n}$ with respect to $R_i$. In Section~\ref{sec:Gross-Kunze} we give a construction of the irreducible representations $V$ of $\Ort{n}$ indexed by tuples $\lambda$, and we give a basis $\phi(e_{T_1}),\ldots,\phi(e_{T_m})$ of $V^{\Ort{n-t}}$, where $T$ ranges over the semistandard tableaux on the Young diagram associated to $\lambda$ containing entries from $1$ to $t$. In this section we show how we can use this to construct bases $e_{\pi,i,1},\ldots,e_{\pi,i,d_{\pi, i}}$ of the spaces $V^{H_i}$ as needed in \eqref{eq:kernelfourier}.

We can choose the representatives $R_i$ to lie in the span of the first $t_i = |R_i|$ coordinates. We assume for notational simplicity that the vectors in $R_i$ are linearly independent, which is true for all applications considered in this paper. We think of $\Ort{n-t_i}$ as acting on the last $n-t_i$ coordinates, so that 
\[
H_i = S(R_i) \oplus \Ort{n-t_i},
\]
where $S(R_i)$ is a finite group of orthogonal transformations that act in the first $t_i$ coordinates and act on $R_i$ by permuting elements. We will abuse notation and think of these transformations as $t_i \times t_i$ matrices, knowing that they act as the identity on the last $n-t_i$ coordinates. We have
\[
V^{H_i} = \big( V \vspace{0mm} ^{\Ort{n - t_i}} \big)  \vspace{0mm} ^{S(R_i)}
\]
and we may first construct a basis of $V^{\Ort{n-t_i}}$ and then a basis of the $S(R_i)$-invariants inside $V^{\Ort{n-t_i}}$.

Given the definitions of the representations $\rho$ and~$\pi$ of $\GL{t}$ and $\Ort{n}$ and recalling $t_i \leq t$, we have
\begin{equation}\label{eq:pirho-unfold}
(\pi(\xi) \phi(e_T)) (x) = \rho(\omega x \xi \epsilon)e_T = \sum_{S}  \prod_{j = 1}^{\lambda_1} \prod_{k=1}^{\mu_j} (\omega x \xi \epsilon)_{S(k,j),T(k,j)}  e_S,
\end{equation}
where the sum runs over all tableaux $S$ with entries from $1$ to $t$ on the Young diagram associated to the partition $\lambda$. This shows we have $\pi(\xi) \phi(e_T) = \phi(e_T)$ for all $\xi \in \Ort{n-t_i}$ if the tableau $T$ only has entries ranging from $1$ to $t_i$, and thus the vectors $\phi(e_T)$ lie inside $V^{\Ort{n-t_i}}$ as we range over such $T$. As before, the subspace $V^{\Ort{n-t_i}}$ has dimension equal to the corresponding representation of $\GL{t_i}$. Hence a basis of $V^{\Ort{n-t_i}}$ is given by the set of $\phi(e_T)$, where $T$ ranges over semistandard tableaux with entries from $1$ to $t_i$. Henceforth we enumerate the tableaux in such a way that the first $m_i$ semistandard tableaux give a basis for $V^{\Ort{n-t_i}}$ and the bases for the spaces $V^{\Ort{n}} \subseteq V^{\Ort{n-1}} \subseteq \dots \subseteq V^{\Ort{n-t}}$ are nested.

To use the basis of $V^{\Ort{n-t_i}}$ to construct a basis for $V^{H_i}$ we use the following procedure. First, we construct the matrices generating the groups $S(R_i)$. For this we fix an ordering of the vectors in $R_i$ and denote also by $R_i$ the $n \times t_i$ matrix with these vectors as columns. The Gram matrix $R_i^{\sf T}R_i$ of $R_i$ is a matrix with ones in the diagonal and elements from $D$ elsewhere. Given a permutation $\sigma$ of the columns, denote by $P_\sigma$ the corresponding $t_i \times t_i$ permutation matrix. If $P_\sigma^{\sf T} R_i^{\sf T}R_iP_\sigma = R_i^{\sf T}R_i$, then there exists a corresponding $r_\sigma \in S(R_i)$ given by
\[
r_\sigma =  R_i P_\sigma (R_i^{\sf T}R_i)^{-1} R_i^{\sf T}.
 \]
and we observe that $r_\sigma$ satisfies $r_\sigma R_i = R_i P_\sigma$. To obtain the generators for the group $S(R_i)$, we take the generators $\sigma$ of the symmetric group on $t_i$ elements which  satisfy $P_\sigma^{\sf T} R_i^{\sf T}R_iP_\sigma = R_i^{\sf T}R_i$ and take the corresponding matrices $r_\sigma$.

To obtain a basis for the space $V^{H_i}$ we look for linear combinations of the vectors $\phi(e_{T_j})$ with $1 \leq j \leq m_i$ such that  
$\pi(r_\sigma) \sum_{j} c_j \phi(e_{T_j}) = \sum_{j} c_j \phi(e_{T_j})$
and in order to do so we compute a basis for the solution space of the system
\begin{equation}\label{eq:kernel-invariant}
\sum_{j=1}^{m_i}\big\langle \phi(e_{T_k}),\  (\pi(r_\sigma) - I) \phi(e_{T_j}) \big\rangle c_j  = 0
\end{equation}
for all $1 \leq k \leq m_i$ and all generators $r_\sigma$ of $S(R_i)$. For each basis element $c$ we get the basis element $\sum_j c_j \phi(e_{T_j})$ of $V^{H_i}$. 

Let us look at an example. Let $t = 2$ and consider an orbit $X$ such that the elements of $X$ have cardinality two. Let $r$ be the $n \times n$ orthogonal matrix  which maps $e_2 $ to $ -e_2$ and fixes the orthogonal complement of $e_2$. We may choose a representative $R = \{ v_1,v_2\}$ of $X$ such that $rv_1 = v_2$ and $rv_2 = v_1$.  Then the group $S(R)$ is the group of two elements $\{ I , r\}$. Unwinding definitions with \eqref{eq:pirho-unfold} and using
\[
\rho(\omega x r \epsilon) = \rho(\omega x \epsilon )\rho \left(
\begin{bmatrix}
1 & 0 \\
0 & -1 \\
\end{bmatrix}  
\right),
\]
we have $\pi(r)\phi(e_T) = \phi(e_T)$ if there is an even number of twos in the tableau $T$ and $\pi(r)\phi(e_T) = - \phi(e_T)$ if there is an odd number of twos in $T$. Then $H = S(R) \oplus \Ort{n-2}$ is the stabiliser of $R$. A basis of $V^{H}$ is given by the set of those $\phi(e_T)$ for which $T$ is a semistandard tableau with an even number of twos.

\section{Semidefinite programming formulation}
\label{sec:sdp}

In this section we give the semidefinite programming formulation. We also discuss a projective formulation for the case of bounding $N_\alpha(n)$, and we give implementation details.

To solve \eqref{pr:las} using semidefinite programming we parametrize the positive kernels using \eqref{eq:kernelfourier}. To do this explicitly we need to compute inner products of the form
\[
\big\langle \pi(s(J_1)) e_{\pi,i(J_1),j_1}, \pi(s(J_2)) e_{\pi,i(J_2),j_2} \big\rangle.
\]
Here $V$ is the irreducible representation of $\Ort{n}$ as constructed in Section~\ref{sec:Gross-Kunze}, and $\{e_{\pi,i,j}\}_j$ is the basis of $V^{H_i}$ as constructed in Section~\ref{sec:stabinv}, where $H_i$ is the stabilizer subgroup of $\Ort{n}$ with respect to the orbit representative $R_i$ of the orbit $X_i$. Recall that $s$ is a function $\mathcal I_t \to \Ort{n}$ such that $s(J)R_{i(J)} = J$ for all $J \in \mathcal I_t$, and define $S = s(J_1)^{-1} s(J_2)$. Then the above inner product is equal to
\[
\big\langle e_{\pi,i_1,j_1}, \pi(S) e_{\pi,i_2,j_2} \big\rangle,
\]
where $i_1 = i(J_1)$ and $i_2 = i(J_2)$. 

As before, we view $\Ort{n - t_{i_1}} \subseteq \Ort{n}$ as a subgroup by extending with an identity block $I_{t_{i_1}}$ in the top left corner and similarly for $\Ort{n - t_{i_2}} \subseteq \Ort{n}$. If $S$ and $S'$ are orthogonal matrices of which the top left $t_{i_1} \times t_{i_2}$ blocks agree, then there exist matrices $h_{i_1} \in \Ort{n - t_{i_1}}$ and $h_{i_2} \in \Ort{n - t_{i_2}}$ such that $h_{i_{1}}Sh_{i_{2}} = S'$. This follows from the following fact: If two tuples of vectors $(v_1,v_2,\dots), (w_1,w_2,\dots)$ with $v_i, w_i \in \mathbb{R}^m$ satisfy $v_i \cdot v_j = w_i \cdot w_j$ for all $i$ and $j$, then there exists an $h \in \Ort{m}$ such that $hv_i = w_i$. By the invariance properties of the basis, we then have 
\[
\big\langle e_{\pi,i_1,j_1}, \pi(S) e_{\pi,i_2,j_2} \big\rangle = \big\langle e_{\pi,i_1,j_1}, \pi(S') e_{\pi,i_2,j_2} \big\rangle.
\]
Hence to calculate $\big\langle e_{\pi,i_1,j_1}, \pi(S) e_{\pi,i_2,j_2} \big\rangle$, one may replace $S$ by any orthogonal matrix $S'$ with the same top left $t_{i_1} \times t_{i_2}$ block.

To find the top left $t_{i_1} \times t_{i_2}$ block of $S$, we first compute $s(J_1)$. Fix an ordering of the vectors in $R_{i_1}$ and $J_1$, and denote also by $R_{i_1}$ and $J_1$ the matrices with these vectors as their columns. Let $A = R_{i_1}^{\sf T}R_{i_1}$ and let $P$ be a permutation matrix for which $A = P^{\sf T} J_1^{\sf T} J_1 P$. Let $Q$ be an orthogonal matrix with the first $t_{i_1}$ columns given by the corresponding columns of
$
J_1 P A^{-1} R_{i_1}^{\sf T}.
$
The representative $R_{i_1}$ is chosen to be supported in the first $t_{i_1}$ coordinates, so we have $Q R_{i_1} = J_1P$, and we can define $s(J_1) = Q$. The matrix $s(J_2)$ is computed similarly. Then the top-left $t_{i_1} \times t_{i_2}$ block of $S = s(J_1)^{-1} s(J_2)$ is given by the corresponding block of 
\begin{equation} \label{eq:t1timest2blockofS}
(J_1 P_{i_1} A_{i_1}^{-1} R_{i_1}^{\sf T})^{\sf T} J_2 P_{i_2} A_{i_2}^{-1} R_{i_2}^{\sf T}.
\end{equation}
In the implementation we construct $S \in O(n)$ as a $2t \times t$ matrix with orthonormal columns of which the top-left $t_{i_1} \times t_{i_2}$ block agrees with the corresponding block of~\eqref{eq:t1timest2blockofS}.

We now define $w_{\pi,i,j}$ by $\phi(w_{\pi,i,j}) = e_{\pi,i,j}$ and evaluate the inner product as follows: 
\begin{align*}
    &\big\langle e_{\pi,i_1,j_1}, \pi(S) e_{\pi,i_2,j_2} \big\rangle\\
    &\quad = \int_{\Ort{n}} \big\langle \phi(w_{\pi,i_1,j_1})(x),  \phi(w_{\pi,i_2,j_2})(xS)\big\rangle \, dx\\
    &\quad = \int_{\Ort{n}} \big\langle \rho(\omega x \epsilon) w_{\pi,i_1,j_1},  \rho (\omega x S \epsilon) w_{\pi,i_2,j_2} \big\rangle\, dx \\
    &\quad = \sum_k \int_{\Ort{n}} \big\langle \rho(\omega x \epsilon) w_{\pi,i_1,j_1}, e_{T_k} \big\rangle \big\langle e_{T_k}, \rho (\omega x S \epsilon) w_{\pi,i_2,j_2} \big\rangle\, dx,
\end{align*}
where the matrix coefficients in the integrand of the last integral can be computed using \eqref{eq:matrixcoeffs}.
Notice that the integrand is a polynomial in the top-left $2t \times 2t$-part of $x$. By using the recursive approach from \cite{gorinlopez08} for evaluating the integral of a monomial over $\Ort{n}$ we can compute the inner product $\langle e_{\pi,i_1,j_1}, \pi(S) e_{\pi,i_2,j_2} \big\rangle$ explicitly as a polynomial in the entries of $S$ where the coefficients are rational functions in the dimension $n$. This shows how we can express the kernel $K$ from \eqref{eq:kernelfourier} explicitly in terms of the matrices $F^\pi$ defining the kernel.

By fixing $d$ and considering representations $\pi$ with $|\lambda| \leq d$ we can now generate the following semidefinite program, where the entries of the constraint matrices are explicitly computed rational functions in the dimensions $n$ over the ground field of the algebraic numbers:
\begin{mini}
{}{K(\emptyset, \emptyset)}{}{}
\label{pr:lassdp}
\addConstraint{F^{\pi} \succeq 0,}{}{\pi = \pi_\lambda, \lambda = (\lambda_1,\ldots,\lambda_t), |\lambda| \leq d}
\addConstraint{A_t K(R) \leq -1_{\mathcal{I}_{=1}}(R),\qquad}{}{R \in \mathcal R_{2t}.}
\end{mini}
In the program, $\mathcal R_{2t}$ is a set containing a representative of each orbit in $\mathcal I_{2t} \setminus \{\emptyset\}$.

To bound $N_\alpha(n)$ we consider the spherical finite distance problem with $D = \{\pm\alpha\}$. In this case we can alternatively obtain upper bounds by considering the graph on the projective space $\mathbb{RP}^{n-1}$ with $[x] \sim [y]$ if $x\cdot y = \alpha$ or $x \cdot y = -\alpha$. This leads to a reduction in the number of variables and constraints. The kernels are defined in the same way as in (\ref{eq:kernelfourier}), but there are some minor differences. Firstly, the invariant space $V^{H_i}$ is unchanged if $|\lambda|$ is even, but vanishes if $|\lambda|$ is odd. This follows from the fact that $-I$ stabilizes every element in the projective space and the formula
\[
\pi(-I) \phi(v)(x)=\rho(- \omega x \epsilon)v = (-1)^{|\lambda|} \rho(\omega x \epsilon)v = (-1)^{|\lambda|} \phi(v)(x) \, .
\]
Secondly, there are fewer orbits, since now we consider two Gram matrices as representatives of the same orbit not only if they are equal up to simultaneous permutations of rows and columns, but also if they are equal up to simultaneous multiplication by $-1$ of a subset of rows and columns. Both of these changes lead to a smaller semidefinite program.   In contrast to the nonprojective version, the resulting semidefinite programs exist over the rational numbers for $t=2$ and the odd integers $a$ we tried. We have observed empirically that for $t = 2$ the projective and nonprojective bounds are the same, but that for $t=3$ and $d > 4$ there exist cases where they are different. 

We wrote the software to set up the semidefinite programs in Julia~\cite{bezanson17} using the Nemo~\cite{fieker17} computer algebra system, where we use Calcium \cite{MR4398788} for exact arithmetic with algebraic numbers.
For solving the semidefinite programs we use the solver SDPA-GMP~\cite{nakata10}. Setting up the semidefinite programs, especially for the third level of the hierarchy, is computationally demanding. We pregenerate the part of the zonal matrices not depending on the dimension $n$, which is fast for $t=2, 3$ and $d\leq 4$, but takes about a week for the $t=3$ and $d=5$ case on a single core of a modern computer. After that, the whole process of setting up and solving the semidefinite programs takes a few seconds for $\las_2$ using degree $4$ polynomials and a few hours for $\las_3$ using degree $5$ polynomials. For the case $t=3$ and $d \ge 4$ the algebraic numbers become complicated which means setting up the semidefinite programs in such a way that the entries of the constraint matrices are rational functions in $n$ is no longer feasible. Hence, for these cases we only generate the semidefinite programs for fixed dimensions. The source code and data files for the proofs can be found in the ancillary files of the arXiv version of this paper, including documentation on how to use the code.

\section{New bounds in fixed dimensions}
\label{sec:applications}

\subsection{Improved bounds on $N_\alpha(n)$ }
\label{sec:equiangular-lines}

There has been extensive research into the determination of $N_\alpha(n)$. As mentioned in the introduction, usually only inner products $\alpha = 1/a$ with $a$ an odd natural number are considered, since these are the only cases for which there may exist equiangular line configurations of size larger than $2n$ \cite{lemmens73}, and the determination of $N_\alpha(n)$ for these values of $\alpha$ solves the equiangular lines problem without fixed angle in dimension $n$ \cite{larman77}.
In \cite{lemmens73, neuimaier89, cao21} the equations
\[
N_{1/3}(n) = \begin{cases} 28 & \text{for } 7 \leq n \leq 14\text{, and}\\
2(n-1) & \text{for } n \geq 15\end{cases}
\]
and
\[
N_{1/5}(n) = \begin{cases} 276 & \text{for } 23 \leq n \leq 185\text{, and}\\
\lfloor \frac{3}{2}(n-1)\rfloor & \text{for } n \geq 185,\end{cases}
\]
are shown, but for $a \geq 7$ less is known (see, e.g., \cite[Appendix A]{kaoyu}). As mentioned in the introduction, the correct values are known for very large dimensions and there exist general upper and lower bounds.

\begin{figure}
\input{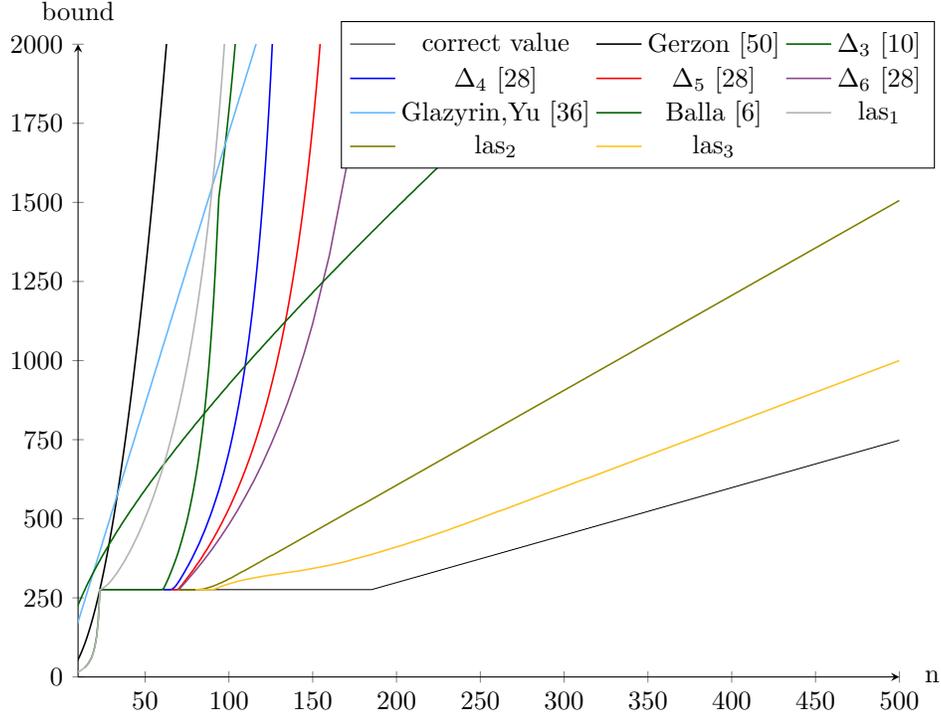}
\caption{Bounds for $N_{1/5}(n)$.}
\label{fig:a5}
\end{figure}

\begin{figure}
\begin{center}
\input{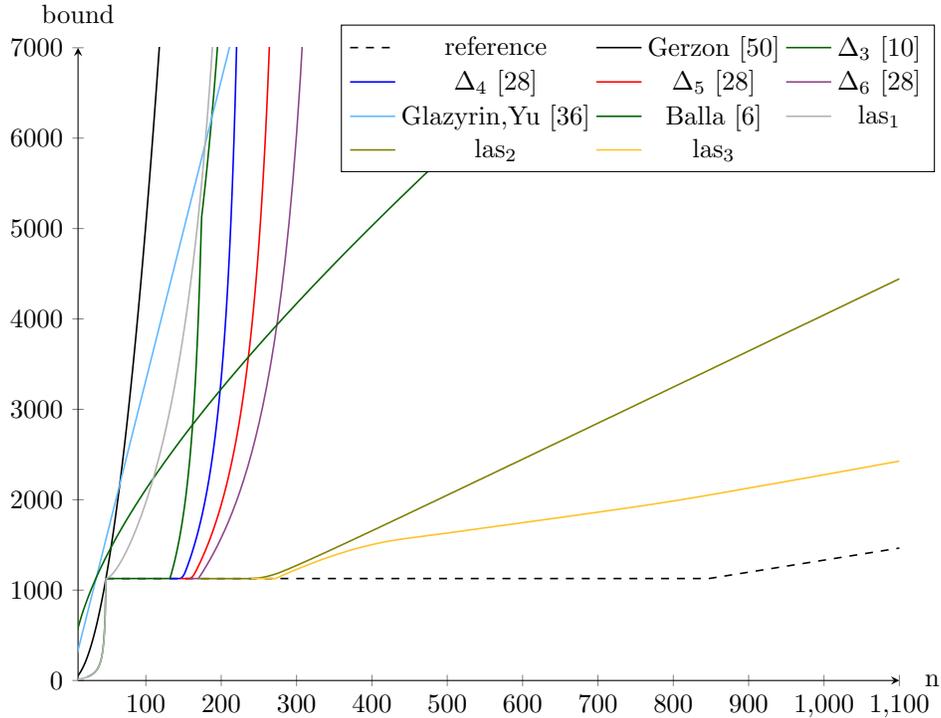}
\end{center}
\caption{Bounds for $N_{1/7}(n)$. The dashed line extends \eqref{eq:stable} until dimension \eqref{eq:dimlinearcon}, from which it continues with the construction of ${\lfloor (n-1)(a+1)/(a-1)\rfloor}$ lines.}
\label{fig:a7}
\end{figure}

A fundamental result in the linear/semidefinite programming approach is the Delsarte, Goethals, and Seidel~\cite{delsarte77} linear programming bound. Since this bound takes into account constraints between pairs of points, it is called a $2$-point bound which we denote by $\Delta_2$. In \cite{bachoc08, barg14} this is generalized to a $3$-point bound $\Delta_3$ and computed for the equiangular lines problem, and in \cite{musin14,deLaat2021} this is generalized to a $k$-point bound $\Delta_k$ and computed for $k=4,5,6$. As mentioned in the introduction we have $\Delta_2 = \las_1$. The bounds $\las_2$ and $\las_3$ considered in this paper thus provide an alternative generalization of the linear programming bound. As shown in Figure~\ref{fig:a7}, for $a = 7$ and for many dimensions the bounds $\las_2$ and $\las_3$ are much stronger than any of the previous bounds coming from semidefinite programming or elsewhere. Similar results hold for larger values of $a$, but we do not include the plots because they look qualitatively similar.

One common feature of $\Delta_k$ for $k = 3,4,5,6$ is that starting at $n = a^2 - 2$ they stabilize and produce the constant bound
\begin{equation}\label{eq:stable}
\frac{(a^2-2)(a^2-1)}{2}
\end{equation}
on $N_{1/a}(n)$ until a certain dimension $D_k(a)$. For $a = 5$ there is an exceptional configuration related to the Leech lattice of $(a^2-2)(a^2-1)/2 = 276$ points in dimension $a^2-2$ \cite{conway69,lemmens73}. This configuration stays optimal until a construction of ${\lfloor (n-1)(a+1)/(a-1)\rfloor}$ equiangular lines in dimension $n$ (see \cite{bukh16}) matches \eqref{eq:stable}, which happens in dimension
\begin{equation}
\label{eq:dimlinearcon}
\frac{(a^2-2)(a-1)^2}{2} + 1.
\end{equation}
The general situation is different, however, since it is known that a configuration of $(a^2-2)(a^2-1)/2$ lines in dimension $a^2-2$ cannot exist for a number of values of $a$ starting at $a=7$ \cite{bannai04, nebe12}. In fact, in \cite{kaoyu} it is shown that for those values of $a$ there is no configuration of $(a^2-2)(a^2-1)/2$ equiangular lines in dimensions $a^2-2 \leq n \leq D_4(a)$.

It is therefore of interest to find better semidefinite programming bounds on $N_{1/a}(a^2 - 2)$ and to increase the range of dimensions for which we know that \eqref{eq:stable} gives an upper bound on $N_{1/a}(n)$. As can be seen in Table~\ref{table:new-ranges}, the bounds $\las_2$ and $\las_3$ are equal to \eqref{eq:stable} for a significantly larger range of dimensions. Furthermore, we have obtained an example where $\las_3$ improves over $\las_2$ and $\las_1$ and hence the linear programming bound for a dimension lower than $a^2 - 2$. For $a = 7$ and $n = 13$ the bound improves from $17$ to $16$. 

In \cite{yu17, kaoyu} explicit quadratic expressions are given for $D_3(a)$ and $D_4(a)$. Based on the data in Table~\ref{table:new-ranges} we conjecture that for $\las_2$ and $\las_3$ the corresponding expressions are  cubic instead of quadratic. Recall that \eqref{eq:dimlinearcon} is quartic in $a$.

\begin{table}
\begin{tabular}{@{}ccccccccc@{}}
\toprule
$a$ & $\Delta_2=\las_1$ & $\Delta_3$ & $\Delta_4$ & 
$\Delta_5$ & $\Delta_6$ & $\mathrm{las}_2$ & $\mathrm{las}_3$ & \text{intersection}\\
\midrule
$5$ & $23$ & $60$ & $65$ & $69$ & $70$ & $82$ $(80)$  & $90$ $(89)$ & $185$\\
$7$ & $47$ &  $131$ & $145$ & $158$ & $169$  & $243$ $(239)$ & $272$ $(271)$ & $847$\\
$9$ & $79$ & $227$ & $251$ & $273$ & $300$ & $535$ $(530)$ & $610$ $(610)$ & $2529$ \\
$11$ & $119$ & $347$ & $381$ & $413$ & $448$ & $1000$ $(993)$ & $1152$ $(1152)$ & $5951$ \\
$13$ & $167$ && & & & $1676$ $(1668)$ & $1946$ $(1946)$ & $12025$ \\
$15$ & $223$ && & & & $2604$ $(2595)$ & $3040$ $(3040)$ & $21855$ \\
$17$ & $287$ && & & & $3823$ $(3813)$ & $4483$ $(4483)$ & $36737$ \\
$19$ & $359$ && & & & $5374$ $(5362)$ & $6321$ $(6321)$ & $58159$ \\
$21$ & $439$ && & & & $7294$ $(7281)$ & $8603$ $(8603)$ & $87801$ \\
$23$ & $527$ && & & & $9626$ $(9611)$ & $11377$ $(11377)$ & $127535$ \\
$25$ & $623$ && & & & $12407$ $(12391)$ & $14692$ $(14692)$ & $179425$ \\
\bottomrule
\end{tabular}
\bigskip
\caption{The largest dimension~$n$ for which the respective bounds can be used to show \eqref{eq:stable} holds. In parentheses we list the largest dimension for which the bound is exactly equal to $(a^2-2)(a^2-1)/2$. The bound $\Delta_3$ is computed in~\cite{barg14, king16} and the bounds $\Delta_4$, $\Delta_5$, $\Delta_6$ are computed in~\cite{deLaat2021}. The column labeled `intersection' shows the dimension from \eqref{eq:dimlinearcon} for reference.}
\label{table:new-ranges}
\end{table}

\subsection{Bounds for more general distance sets}
\label{sec:quasi-unbiased}

In this section we discuss some applications of $\las_2$ to problems where the allowed distance set $D$ is not of the form $\{\pm \alpha\}$. For some of these computations the cardinality of $D$ is $3$ as opposed to $2$. Here nothing changes in the formulation of the bounds, but the number of orbits to consider increases greatly and $\las_3$ becomes too hard to compute. For instance, with inner products $\{1/7, -1/7\}$  there are $156$ orbits with sets of size $6$, while with inner products $\{1/7, -1/7, 0\}$ there are $25506$ such orbits. For completeness we also mention applications we tried where we did not get new results. 

The first application we consider is to a problem with sets of matrices having orthogonal rows. Following~\cite{araya17, kao22+}, a Hadamard matrix of order $n$ is a $(+1,-1)$-valued $n \times n$ matrix $H$  such that $HH^{\sf T} = n I$, and a weighing matrix of order $n$ and weight $k$ is a $(+1,-1,0)$-valued $n \times n$ matrix such that $WW^{\sf T}=k I$. Two Hadamard matrices $H_1$ and $H_2$ of order $n$ are said to be quasi-unbiased for parameters $(l,a)$ if $a^{-1/2}H_1 H_2^{\sf T}$ is a weighing matrix of weight $l$. Note that necessarily $l = n^2/a$.
Hadamard matrices $H_1,\ldots,H_f$ of order $n$ are said to be quasi-unbiased for parameters $(l,a)$ if they are pairwise quasi-unbiased for those parameters. 
Table 1 of Araya, Harada, and Suda~\cite{araya17} has a list of the possible parameters for quasi-unbiased Hadamard matrices of order up to $48$ together with bounds for the maximum size of these sets. Kao, Suda, and Yu~\cite{kao22+} applied the semidefinite programming bound $\Delta_3$ to derive bounds based on an observation that normalizing the rows of a set of $f$ quasi-unbiased Hadamard matrices for parameters $(l, a)$ gives a spherical $3$-distance set in $S^{n-1}$ with $|X| = nf$ and inner products $\{\pm l^{-1/2}, 0\}$. 
We apply $\las_2$ using polynomial representations of degree $|\lambda| \leq 8$ to give the bounds as listed in Table~\ref{table:QUHM}.

\begin{table}
\begin{tabular}{@{}ccccccc@{}} 
\toprule
$n$ & $l$ & $a$ & Lower bound from~\cite{araya17} & Upper bound from~\cite{araya17} & $\Delta_3$~\cite{kao22+} & $\mathrm{las}_2$\\ 
\midrule 
$16$ & $4$ & $64$  & $8$ & $35$ & $15$  & $15$ \\  
$24$ & $4$ & $144$ & $2$ & $85$ & $25$ & $23$ \\ 
$24$ & $9$ & $64$  & $16$ & $85$ & $95$ & $95$ \\
$32$ & $4$ & $256$ &  $8$ & $155$ & $47$ & $31$ \\
$36$ & $9$ & $144$ & -& $199$ & $79$ & $ 67$ \\
$40$ & $4$ & $400$ & -& $247$ & $101$ & $39$ \\ 
$40$ & $25$& $64$ & -& $28$ & $30$ & $30$ \\ 
$48$ & $4$ & $576$ & $2$ & $361$ & $276$ & $47$ \\ 
$48$ & $9$ & $256$ & $16$ & $361$ & $104$ & $96$ \\ 
$48$ & $16$ & $144$ & -& $361$ &  $316$ & $316$ \\ 
$48$ & $36$ & $64$ & $2$ & $28$ &  $30$ & $30$ \\ 
\bottomrule
\end{tabular}
\bigskip
\caption{Improved bounds on the maximum number of quasi-unbiased Hadamard matrices of order $n$ with parameters $(l,a)$.}
\label{table:QUHM}
\end{table}

We applied our bound to a problem involving ``Q-antipodal Q-polynomial schemes with 3 classes'' as described by Martin, Muzychuk, Williford~\cite{martin07}, but we observed that $\las_2$ produces the same results as listed in the table in~\cite{kao22+}.

Another application of bounds for $3$-distance sets is to the maximum size of a $3$-distance set in $S^{n-1}$ for any possible choice of angles \cite{musin11,szollosi20}. Here a theorem by Nozaki~\cite{nozaki11} is used so that for each inner product $d_3$ only finitely many other inner products $d_1$ and $d_2$ need to be considered. By using sums-of-squares techniques the bounds $\Delta_2 = \las_1$ and $\Delta_3$ can be applied \cite{musin11,szollosi20,lin20+}. We applied $\las_2$ to this problem, and did get improvements for some parameters, but overall could not get new results because like $\Delta_2$ and $\Delta_3$, $\las_2$ becomes unbounded as $d_3 \to 1$.

Finally, we tried to disprove the existence of certain strongly regular graphs; see Cameron~\cite[Chapter 8]{beineke04} for an introduction to these graphs and also the more complete monograph by Brouwer and Van Maldeghem~\cite{brouwer22}. The existence of a strongly regular graph with a given set of parameters implies the existence of a spherical $2$-distance set in a certain dimension and with certain inner products; see Theorem~5.1 in~\cite[Chapter 8]{beineke04}.
We applied $\las_3$ with $|\lambda| \leq 5$ to all sets of parameters listed as open in~\cite[Chapter 12]{brouwer22} with at most $200$ vertices, without getting a new result.

\section{Asymptotic analysis of the bounds}
\label{sec:asymptotics}

In this section we explain how we obtain a computer generated and verified proof of Conjecture~\ref{conj:las2} for $a = 3,5,7,9,11$. 

First we observe empirically that the degree of the required representations does not grow. In fact, for large dimensions, we only use the representations with $\lambda = (0,0), (3,1), (4,0)$. Though it is not necessary for our proof, we have the following conjecture:

\begin{conjecture}\label{con:lambdas}
The optimal value of $\mathrm{las}_2$ for bounding $N_{\alpha}(n)$ can be obtained by using only the representations with $\lambda = (0,0), (2,0), (3,1), (4,0)$. Moreover, for dimensions beyond the stable range the representation with $\lambda = (2,0)$ is not needed.
\end{conjecture}

As numerical evidence for Conjecture~\ref{conj:las2} we observed that the function $\las_2(n)$ can be expanded in terms of $n, 1, n^{-1}, n^{-2}, \dots$, and through interpolation we find the first few coefficients in this expansions for many values of $\alpha$. We found that the first expansion coefficient satisfies the formula $(1 + \alpha)/(2\alpha)$. For $\alpha = 1/5$, the expansion of $\las_{2}(n)$ seems to be particularly well-behaved having rational coefficients:
\begin{equation}\label{eqn:las2expansion}
3n + 6 + \frac{120}{n} + \frac{5530}{n^2} +  \frac{1449485}{3n^3} + \frac{2961283225}{72n^4} + O\left(\frac{1}{n^5}\right).
\end{equation}

To prove Conjecture~\ref{conj:las2} for a given value of $\alpha$ we construct, for each sufficiently large $n$, a feasible solution to the semidefinite program $\mathrm{las}_2(n)$ such that the corresponding sequence of objective values is linear in $n$ with slope $(1 + \alpha)/(2\alpha)$. For this we consider the perturbed hierarchy $\mathrm{las}_{2,4}(n)$, where we subtract $1/n^4$ from the right hand side of each inequality constraint in \eqref{pr:lassdp} and force each eigenvalue of each block matrix to be at least $1/n^4$. Let $\{F^{\lambda}(n)\}_\lambda$ be the optimal solution of $\mathrm{las}_{2,4}(n)$ lying on the central path of the interior-point method. We now make the ansatz that there exist matrices $A^{\lambda,k}$, whose entries are algebraic numbers of low degree and reasonable bit size, such that 
\begin{equation}\label{eq:expansion}
F^{\lambda}(n)_{(i_1,j_1),(i_2, j_2)} = \sum_{k=0}^\infty A_{(i_1,j_1), (i_2,j_2)}^{\lambda,k} n^{1+\lambda_1+2\lambda_2  - t_{i_1} - t_{i_2}-k},
\end{equation}
where, as before, $t_i$ is the cardinality of the orbit representative $R_i$.
Then we use the interior-point solver to numerically compute a near optimal solution approximately on the central path of $\mathrm{las}_{2,4}(n)$ for dimensions $N, N+1, \dots, N+L$, and we use this to compute approximations of the coefficient matrices $A^{\lambda,0},\ldots,A^{\lambda,l-1}$ via interpolation.  We then use the LLL algorithm to find the entries of the coefficient matrices exactly as algebraic numbers, and denote by $\mathrm{sol}(n)$ the solution whose matrices are given by the truncation of \eqref{eq:expansion} using these $l$ rounded coefficient matrices. 

We found good values for the parameters $l$, $L$, and $N$ through experimentation. If the dimension $N_\alpha$, beyond which the solutions are feasible, is to be made small, the number of terms $l$ in the truncation should not be too small and not too large. Perhaps this is due to Runge's phenomenon. After that, $N$ and $L$ should be chosen such that we  find $A^{\lambda,0},\ldots,A^{\lambda,l-1}$ in sufficiently high precision so that we can round them correctly. For the results presented in this paper we use $N = 10^{100}$ and the parameters $l$ and $L$ listed in Table~\ref{table:asymptoticparameters}.

\begin{table}
\begin{tabular}{@{}cccccccc@{}} 
\toprule
$\alpha$ & $f_\alpha(n)$ & $N_\alpha$ & $l$ & $L$ & $N_{\alpha,4}$ & $N_{\alpha,5}$\\
\midrule 
$1/3$ & $2n+4$ & $500$ & $9$ & $12$ & $17$ & $12$  \\
$1/5$ & $3n+30$ & $2235$ & $10$ & $15$ & $253$ & $87$\\
$1/7$ & $4n+116$ & $13739$ & $9$ & $11$  & $4638$ & $261$\\
$1/9$ & $5n+316$ & $166018$ & $9$ & $12$ \\
$1/11$ & $6n+699$ & $751307$ & $9$ & $12$ \\
\bottomrule
\end{tabular}
\bigskip
\caption{Parameters used to obtain the interpolations.}
\label{table:asymptoticparameters}
\end{table}

Next, we verify that $\mathrm{sol}(n)$ is a solution for $\mathrm{las}_2(n)$. Since the entries of both $\mathrm{las}_2(n)$ and $\mathrm{sol}(n)$ are exact rational functions in $n$, we can compute the slack in the inequality constraints of the solution $\mathrm{sol(n)}$ to $\mathrm{las}_2(n)$ as exact rational functions in $n$ (where a positive slack means the inequality constraint is satisfied strictly). We then fix an integer $n_\alpha$ and evaluate these rational functions at $n_\alpha+n$. We then verify that the coefficients of the numerator and denominator polynomials are all positive, which proves the slacks are positive for $n \geq n_\alpha$, and hence the inequality constraints are satisfied for all $n \geq n_\alpha$. Then we compute the determinants of the leading principal submatrices, evaluate these rational functions in $n_\alpha+n$, and check that all coefficients of the numerator and denominator polynomials are positive, which proves the solution matrices are positive semidefinite for all $n \geq n_\alpha$. Finally, we check that the objective function is linear in $n$ with slope $(\alpha +1)/(2\alpha)$, which gives a computer verified proof of Conjecture~\ref{conj:las2} for this value of $\alpha$. Note that although floating point computations are used to obtain the proofs, the verification procedure is implemented entirely in exact arithmetic.

To make the value $n_\alpha$, beyond which we can prove our bound holds, as small as possible we solve finitely many semidefinite programs in fixed dimensions. We only do this for the values $a = 3,5,7$, but in principle it could be done for $a = 9,11$ too. First, we solve $\las_{2,4}(n)$ or $\las_{2,5}(n)$. Then, we approximate the floating point solution by a rational solution and we check in exact arithmetic whether the rounded solution is feasible for $\las_2(n)$ and has objective below $f_\alpha(n)$. The reason we use two different perturbations is that $\las_{2,4}(n)$ does not give good enough bounds in low dimensions and it is too difficult to find a feasible solution for $\las_{2,5}(n)$ in high dimensions. We use $\las_{2,4}(n)$ for $N_{\alpha,4} \leq n < N_\alpha$ and $\las_{2,5}(n)$ for $N_{\alpha,5} \leq n \leq N_{\alpha, 4}$. 

For $\alpha = 1/7$, for example, the interpolation procedure shows that
\[
\las_2(n) \leq 4n + a_1 + a_2 n^{-1} + \dots + a_8 n^{-7} \quad \text{for all} \quad n \geq 13739,
\]
for certain explicitly given $a_1,\ldots,a_8 \in \mathbb Q[\sqrt{2}]$. From this we can then derive that $N_{1/7}(n) \leq 4n+116$ for all $n \geq 13739$. Next we solve finitely many semidefinite programs to decrease the dimension $N_\alpha = 13739$ to $n_\alpha = 261$. In Table~\ref{table:asymptotic} we list the bounds obtained with this approach. Note that the same approach works, in principle, for other values of $\alpha$ not listed in the table, but we did not perform these computations.

For $t=3$ we seem to get asymptotically linear bounds with a better slope than with $t=2$. However, since computing the third level of the hierarchy for $d > 5$ is currently too computationally demanding we do not have the equivalent of Conjecture~\ref{con:lambdas}. The bounds might very well improve as we increase $d$ beyond $5$, and we do not know  the slope of the asymptotically linear behavior as the degree $d$ goes to infinity. In Table~\ref{table:las3} we give the numerically computed slopes for $d=4$ and $d=5$. 

\begin{table}
\begin{center}{
\begin{tabular}{@{}ccccccccccc@{}} 
\toprule
$a$ & $\tfrac{a+1}{2}$ & $d=4$ & $d=5$ & $\tfrac{a+1}{a-1}$ & & $a$ & $\tfrac{a+1}{2}$ & $d=4$ & $d=5$ & $\tfrac{a+1}{a-1}$  \\
\cmidrule{1-5} \cmidrule{7-11}
$5$ & $3$ & $2.000$ & $2.000$ & $1.500$ & & 
$19$ & $10$ & $6.948$ & $4.156$ & $1.111$\\
$7$ & $4$ & $2.003$ & $2.003$ & $1.333$ & & 
$21$ & $11$ & $7.975$ & $4.773$ & $1.100$\\
$9$ & $5$ & $2.428$ & $2.065$ & $1.250$ & & 
$23$ & $12$ & $9.018$ & $5.428$ & $1.091$\\
$11$ & $6$ & $3.171$ & $2.268$ & $1.200$ & & 
$25$ & $13$ & $10.071$ & $6.117$ & $1.083$\\
$13$ & $7$ & $4.038$ & $2.617$ & $1.167$ & & 
$27$ & $14$ & $11.130$ & $6.836$ & $1.077$\\
$15$ & $8$ & $4.968$ & $3.066$ & $1.143$ & & 
$29$ & $15$ & $12.193$ & $7.583$ & $1.071$\\
$17$ & $9$ & $5.943$ & $3.585$ & $1.125$ & &
$31$ & $16$ & $13.258$ & $8.354$ & $1.067$\\
\bottomrule
\end{tabular}}
\end{center}
\bigskip
\caption{Approximate slopes of $\las_3$ for degrees $d=4$ and $d=5$ and inner products $\alpha=1/a$ together with the slope $(a+1)/2$ given by $\las_2$ and the correct asymptotic slope $(a+1)/(a-1)$ proven by~\cite{MR4334975}.}
\label{table:las3}
\end{table}

\section{The limit semidefinite program}
\label{sec:moreasymptotics}

In this section we give a formulation for the limit semidefinite program. To do so we prove a fact about the asymptotic behavior of the Gross-Kunze construction. 

Let $\langle u, v\rangle_{\U{t}}$ be the unique (up to positive scalars) inner product on $W$ such that the restriction of $\rho$ to the compact group $\U{t}$ is a unitary representation. The following conjecture describes the asymptotic behavior of the inner product defined in Section~\ref{sec:Gross-Kunze}.

\begin{conjecture} \label{conj:zonalofid_for_large_n}
For each $\lambda$ there exists a strictly positive scalar $c$ such that 
\[
n^{|\lambda|} \langle \phi(u), \phi(v) \rangle = c  \langle u, v \rangle_{\U{t}} + O(n^{-1})
\]
for all $u, v \in W$.
\end{conjecture} 

We prove this conjecture for $t = 2,3$ and $|\lambda| \leq 4$. First, we set up a semidefinite program for which the matrix $M$ defined by $M_{i,j} = \langle e_{T_i}, e_{T_j} \rangle_{\U{t}}$ is the unique (up to positive scalars) solution. Consider the differential 
\[
d\rho \colon \mathfrak{gl}(t) \to \mathfrak{gl}(W)
\]
of \eqref{eq:rhorep}, which is a representation of Lie algebras. Here $\mathfrak{gl}(t)$ and $\mathfrak{gl}(W)$ are the Lie algebras of all complex linear endomorphisms of $\C^t$ and $W$, respectively. We now give the matrix coefficients of $d\rho$. Because $d\rho$ is linear, it suffices to compute the matrix coefficients in the basis $E_{r,s} = e_r e_s^{\sf T}$. We have $E_{r,s} e_k = \delta_{s k} e_r$ and hence
\[
d\rho(E_{r,s}) e_T = \sum_{S \in D_{r,s}(T)} e_S,
\]
where the sum ranges over the set $D_{r,s}(T)$ of all tableaux $S$ which may be obtained from $T$ by changing exactly one $s$ to an $r$. 
The matrix coefficients of the representation are then given by
\[
\langle e_{T_i}, d \rho (E_{r, s}) e_{T_j} \rangle = \sum_{S\in D_{r,s}(T_j)} \langle e_{T_i}, e_S \rangle.
\]

By uniqueness of the inner product, $M$ is the unique (up to positive scalars) positive definite matrix satisfying the equation
\[
\rho(x)^* M \rho(x) = M
\]
for all $x \in U(t)$. By differentiating, this condition  implies
 \begin{equation} \label{eq:Mliealgcond}
     d\rho(X)^* M = - M d\rho(X)
 \end{equation}
 for all $X \in \dU{t}$. 
Using the exponential map and connectedness of the unitary group, this condition is sufficient too. By linearity it suffices to enforce this condition on a basis of  the skew-hermitian complex matrices $\dU{t}$. 
The explicit formula for the matrix coefficients of $d \rho$ and an explicit choice of basis of $\dU{t}$ give a practical way of checking whether a given matrix defines the inner product $\langle \cdot, \cdot \rangle_{U(t)}$. 

We now consider the asymptotic expansion of $n^{|\lambda|} \langle \phi(u), \phi(v) \rangle$. Let $a$ be a multi-index with $|a| = 2|\lambda|$ and let $x^a = \prod_{i,j=1}^n x_{ij}^{a_{ij}}$ be the corresponding monomial on $\Ort{n}$. For fixed $a$, Theorem 4.3 in \cite{banica2010orthogonal} gives the asymptotic expansion:
\[
\int_{\Ort{n}} x^a\, dx = n^{-|\lambda|}\sum_{k= 0}^\infty H_k(a)n^{-k},
\]
where $H_k(a)$ is a certain combinatorial quantity related to the Brauer algebra. This gives an expansion of the form
\[
\langle \phi(e_{T_i}), \phi(e_{T_j}) \rangle = n^{-|\lambda|} \sum_{k=0}^\infty c^k_{ij} n^{-k},
\]
for certain integers $c^k_{ij}$. We wrote code to calculate the quantity $H_0 (a)$ and hence the leading coefficient matrix $(c^0_{ij})$. For $t = 2,3$ and $|\lambda| \leq 4$, we verified in exact arithmetic that this matrix $(c^0_{ij})$ is a positive definite solution to the system (\ref{eq:Mliealgcond}).
As a sidenote, we suspect that this asymptotic expansion is related to the expansion (\ref{eqn:las2expansion}) and ultimately to our asymptotic solutions in Section~\ref{sec:asymptotics}.

We now give a formulation for the limit semidefinite program.  In our current set-up for $\textup{las}_t$, we use the Gross-Kunze construction with a representation of $\GL{t}{}$. However, we could have used the construction with $\GL{2t}{}$. In this case, we could still locate the $\Ort{n-t}$-invariants using the argument from Section \ref{sec:stabinv}. Using the Gross-Kunze construction with $2t$ is computationally more expensive, since there are more variables, but it has the following property. Let $S = \textup{diag}(s,I_{n - 2t})$ be an orthogonal matrix which is block-diagonal with the first block $s$ of size $2t \times 2t$. All matrices $S$ occurring in $\las_t$ can be chosen to be of this form. We then have
\begin{align*}
     \langle \phi(e_{T_i}), \pi(S) \phi(e_{T_j}) \rangle & = \int_{\Ort{n}} \langle \rho(\omega x \epsilon ) e_{T_i}, \rho(\omega x S \epsilon )e_{T_j} \rangle \,dx \\
    & = \int_{\Ort{n}} \langle \rho(\omega x \epsilon ) e_{T_i}, \rho(\omega x \epsilon s)e_{T_j} \rangle \,dx \\
    & = \sum_k \left( \int_{\Ort{n}} \langle \rho(\omega x \epsilon ) e_{T_i}, \rho(\omega x \epsilon ) e_{T_k} \rangle\, dx \right) \langle e_{T_k}, \rho(s) e_{T_j} \rangle \\
    & = \sum_k \langle \phi(e_{T_i}), \phi(e_{T_k}) \rangle  \langle e_{T_k}, \rho(s) e_{T_j} \rangle.
\end{align*}
In short, if $2t$ is used in the Gross-Kunze construction, then a separation of variables occurs where all dependence on $n$ is in the  $\langle \phi(e_{T_i}), \phi(e_{T_k}) \rangle$ term. Hence the limit semidefinite program as $n \to \infty$ can be written in terms of the inner product~$\langle \cdot, \cdot \rangle_{U(t)}$.

\section*{Acknowledgements}

We thank Christine Bachoc, Anurag Bishnoi, Henry Cohn, Nando Leijenhorst, Giulia Montagna, Fernando Oliveira, and Frank Vallentin for helpful discussions.

\providecommand{\bysame}{\leavevmode\hbox to3em{\hrulefill}\thinspace}
\providecommand{\MR}{\relax\ifhmode\unskip\space\fi MR }
\providecommand{\MRhref}[2]{%
  \href{http://www.ams.org/mathscinet-getitem?mr=#1}{#2}
}
\providecommand{\href}[2]{#2}

\end{document}